\documentclass[12pt]{amsart}
\usepackage{array,CJK}
\usepackage{curves}
\usepackage{amstext}
\usepackage{amsfonts}
\usepackage{amsmath}
\usepackage{amssymb}
\def\cal{\mathcal}

\newtheorem{theorem}{Theorem}[section]
\newtheorem{corollary}[theorem]{Corollary}
\newtheorem{lemma}[theorem]{Lemma}

\theoremstyle{definition}
\newtheorem{definition}[theorem]{Definition}
\newtheorem{example}[theorem]{Example}
\theoremstyle{remark}
\newtheorem{remark}[theorem]{Remark}

\title[Disjoint transitivity on Orlicz spaces]
{Disjoint topological transitivity for weighted translations on Orlicz spaces}

\author[C-C. Chen]{Chung-Chuan Chen}
\address{Department of Mathematics Education, National Taichung University of Education, Taichung 403, Taiwan}
\email{chungchuan@mail.ntcu.edu.tw}

\author[M. Kosti\' c]{Marko Kosti\' c}
\address{Faculty of Technical Sciences, University of Novi Sad, Trg D. Obradovi\' ca 6, 21125 Novi Sad, Serbia}
\email{marco.s@verat.net}


\date{\today}
\begin{document}

\maketitle

\begin{abstract}
Let $G$ be a locally compact group, and let $\Phi$ be a Young function.
In this paper, we give a sufficient and necessary condition for weighted translations on the Orlicz space $L^\Phi(G)$
to be disjoint topologically transitive. This characterization for disjoint chaos follows from the investigation on disjoint transitivity
immediately.
\end{abstract}

\baselineskip17pt

\section{introduction}

About one decade ago, Bernal-Gonz\'alez, B\`es and Peris introduced new notions of linear dynamics, namely, disjoint transitivity and disjoint hypercyclicity in \cite{bg07} and \cite{bp07} respectively. Since then, disjoint transitivity and disjoint hypercyclicity were studied intensely by many authors \cite{bm12,bmp11,bmps12,bms14,kostic,kosticbook,kostic15,ma10,sa11,sa13,shk10,ss14,vlac}.
Indeed, the existence of disjoint hypercyclic operators on separable, infinite-dimensional topological vector spaces was investigated by Shkarin and Salas in \cite{shk10,ss14} independently. B\`es, Martin and Peris studied the disjoint hypercyclicity of composition operators on spaces of holomorphic functions in \cite{bm12,bmp11}. The characterizations for weighted shifts and powers of weighted shifts on $\ell^p(\Bbb Z)$ to be disjoint hypercyclic and supercyclic were demonstrated in \cite{bms14,bp07,ma10} respectively.
In addition, the necessary and sufficient condition for sequences of operators, which map a holomorphic function to a partial sum of its Taylor expansion, to be disjoint universal was given by Vlachou in \cite{vlac}. Kosti\'c also studied disjoint hypercyclicity of $C$-distribution cosine functions and semigroups in \cite{kostic,kostic15}.

In the investigation of disjoint dynamics and disjoint hypercyclicity, the weighted shifts on $\ell^p(\Bbb{Z})$ play an important role to provide concrete examples, which can be viewed as special cases of weighted translations on the Lebesgue space of locally compact groups. Hence the disjoint hypercyclicity on locally compact groups was investigated in \cite{chen171,chen172,hl16}, which subsumes some works of disjoint dynamics on the discrete group $\Bbb Z$ in \cite{bms14,bp07,ma10}. On the other hand, the study on linear chaos and hypercyclicity is recently extended from the Lebesgue space of locally compact groups to the Orlicz space in \cite{aa17,chen-orlicz}. The Orlicz space is a type of important function space generalizing the Lebesgue space. Based on and inspired by these works, our aim in this paper is to characterize disjoint chaos and disjoint transitivity for weighted translations on the Orlicz space of locally compact groups.

These new notions, {\it disjoint topological transitivity} and {\it disjoint hypercyclicity}, are kind of generalizations of transitivity and hypercyclicity respectively. An operator $T$ on a separable Banach space $X$ is called {\it hypercyclic} if there exists $x\in X$ such that the orbit of $x$ under $T$, denote by $Orb(T,x):={\{T^nx:n\in\Bbb{N}\}}$, is dense in $X$.
 It is known that hypercyclicity is equivalent to topological transitivity on $X$. An operator $T$ is
{\it topologically transitive} if given two nonempty open subsets $U,V\subset X$, there is some $n\in\Bbb{N}$ such that $T^n(U)\cap V\neq\emptyset$.
If $T^n(U)\cap V\neq\emptyset$ from some $n$ onwards, then $T$ is called {\it topologically mixing}. If $T$ is topologically
transitive and the set of periodic elements of $T$ is dense in $X$, then $T$ is {\it chaotic}.
We first recall some definitions of disjointness in \cite{bg07,bp07} for further discussions.

\begin{definition}
Given $L\geq2$, the operators $T_1,T_2,\cdot\cdot\cdot,T_L$ acting on a separable Banach space $X$ are {\it disjoint hypercyclic}, or {\it diagonally hypercyclic} (in short, {\it d-hypercyclic}) if there is some vector $(x,x,\cdot\cdot\cdot,x)$ in the diagonal of $X^L=X\times X\times \cdot\cdot\cdot\times X$ such that
$$\{(x,x,\cdot\cdot\cdot,x),(T_1x,T_2x,\cdot\cdot\cdot,T_Lx),(T^2_1x,T^2_2x,\cdot\cdot\cdot,T^2_Lx),\cdot\cdot\cdot\}$$
is dense in $X^L$; if this is the case, then we say $x\in X$ is a d-hypercyclic vector associated to the operators $T_1,T_2,\cdot\cdot\cdot,T_L$.
\end{definition}

For topological dynamics, several new notions were given accordingly in \cite{bp07} as follows:

\begin{definition}
Given $L\geq2$, the operators $T_1,T_2,\cdot\cdot\cdot,T_L$ on a separable Banach space $X$ are {\it disjoint topologically transitive} or {\it diagonally topologically transitive} (in short, {\it d-topologically transitive}) if given nonempty open sets $U,V_1,\cdot\cdot\cdot,V_L\subset X$, there is some $n\in\Bbb{N}$ such that
$$\emptyset\neq U\cap T_1^{-n}(V_1)\cap T_2^{-n}(V_2)\cap\cdot\cdot\cdot\cap T_L^{-n}(V_L).$$
If the above condition is satisfied from some $n$ onwards, then $T_1,T_2,\cdot\cdot\cdot,T_L$ are called {\it disjoint topologically mixing}  (in short, {\it d-topologically mixing}).
\end{definition}

For a single operator, hypercyclicity and topological transitivity are equivalent. However, this is not the case for disjointness.
Indeed, in \cite[Proposition 2.3]{bp07}, the operators $T_1,T_2,\cdot\cdot\cdot,T_L$ are d-topologically transitive if, and only if, $T_1,T_2,\cdot\cdot\cdot,T_L$ have a dense set of d-hypercyclic vectors.

Linear chaos has attracted a lot of attention during last three decades. For instance, linear chaos on $\ell^p(\Bbb{Z}), L^p(\Bbb{R})$ and $L^p(\Bbb{C})$ spaces, and other spaces were studied in \cite{bb09,BBCP,chen141,ge00,kalmes07,kostic09,mp10}. To study linear dynamics systematically, we refer to these classic books \cite{bmbook,gpbook,kosticbook}. Following the idea in \cite{bp07}, the investigation of disjoint chaos was initiated recently in \cite{ckpv}. We formulate the definition of disjoint chaos below.

\begin{definition}
Given $L\geq2$, the operators $T_1,T_2,\cdot\cdot\cdot,T_L$ on a separable Banach space $X$ are {\it disjoint chaotic} or {\it diagonally chaotic} (in short, {\it d-chaotic}) if they are disjoint transitive and the set of periodic elements, denoted by
${\cal P}(T_1,T_2,\cdot\cdot\cdot,T_L)=\{(x_1,x_2,\cdot\cdot\cdot,x_L)\in X^L: \exists\ n \in {\Bbb N}\ \mbox{with}
\ (T_1^nx_1,T_2^nx_2,\cdot\cdot\cdot,T_L^nx_L)=(x_1,x_2,\cdot\cdot\cdot,x_L)\},$ is dense in $X^L$.
\end{definition}

\begin{remark}\label{remark}
According to the definition of disjoint chaos above, it should be noted that the operators $T_1,T_2,\cdot\cdot\cdot,T_L$ are disjoint chaotic if, and only if, they are disjoint transitive and each of them is chaotic separately.
\end{remark}

In this paper, we will mainly give a necessary and sufficient condition for weighted translation operators on the Orlicz space of a locally compact group
to be disjoint topologically transitive. We introduce Orlicz spaces briefly to begin our study.
A continuous, even and convex function $\Phi:{\Bbb R}\rightarrow {\Bbb R}$ is called a {\it Young function} if it satisfies $\Phi(0)=0$, $\Phi(t)>0$ for $t>0$, and $\lim_{t\rightarrow\infty}\Phi(t)=\infty$.
The complementary function $\Psi$ of a Young function $\Phi$ is defined by
$$\Psi(y)=\sup\{x|y|-\Phi(x):x\geq 0\}$$
for $y\in\Bbb{R}$, which is also a Young function. If $\Psi$ is the complementary function of $\Phi$, then
$\Phi$ is the complementary function of $\Psi$, and the Young inequality
$$xy\leq \Phi(x)+\Psi(y)$$
holds for $x,y\geq0$.
Let $G$ be a locally compact group with identity $e$ and a right Haar measure $\lambda$. Then for a Borel function $f$, the set
$$L^\Phi(G)=\left\{f:G\rightarrow \Bbb{C}: \int_G\Phi(\alpha|f|)d\lambda<\infty\ \text{for some $\alpha>0$} \right\}$$
is called the {\it Orlicz space}, which is a Banach space under the Luxemburg norm $N_\Phi$,
defined for $f\in L^\Phi(G)$ by
$$N_\Phi(f)=\inf\left\{k>0:\int_G\Phi\left(\frac{|f|}{k}\right)d\lambda\leq1\right\}.$$
The Orlicz space is a generalization of the usual Lebesgue space.
Indeed, let $\Phi(t)=\frac{|t|^p}{p}$, then the Orlicz space $L^\Phi(G)$ is the Lebesgue space $L^p(G)$.
Over the last several decades, the important properties and interesting structures of Orlicz spaces have been investigated intensely
by many authors. For instance, Piaggio in \cite{pi17} considered Orlicz spaces and the large scale geometry of Heintze groups.
Also, Tanaka recently studied the properties $(T_{L^\Phi})$ and $(F_{L^\Phi})$ for Orlicz spaces $L^\Phi$ in \cite{ta17}.
Weighted Orlicz algebras on locally compact groups were introduced and investigated in \cite{oo15}, which generalized the group algebras.
Hence it is nature to study linear chaos on the wider setting of Orlicz spaces. For more recent works and the textbooks on Orlicz spaces, we refer to \cite{cl17,ha15,rr91}.

It was showed independently in \cite{an97,bg99} that
a Banach space admits a hypercyclic operator if, and only if, it is separable and infinite-dimensional.
Hence we assume that $G$ is second countable and $\Phi$ is $\Delta_2$-regular throughout.
A Young function $\Phi$ is $\Delta_2$-{\it regular} in \cite{rr91} if there exist a constant $M>0$ and $t_0>0$ such that $\Phi(2t)\leq M\Phi(t)$ for $t\geq t_0$ when $G$ is compact, and $\Phi(2t)\leq M\Phi(t)$ for all $t>0$ when $G$ is noncompact. For example, the Young functions $\Phi$ given by
$$\Phi(t)=\frac{|t|^p}{p}\quad(1\leq p<\infty),\qquad \mbox{and}\qquad\Phi(t)=|t|^\alpha(1+|\log|t||)\quad (\alpha>1)$$
are both $\Delta_2$-regular in \cite{aa17,rr91}. If $\Phi$ is $\Delta_2$-regular, then the space $C_c(G)$ of all continuous functions on $G$ with compact support is dense in $L^\Phi(G)$.

A bounded continuous function $w:G \rightarrow (0,\infty)$ is called a {\it weight} on $G$.
Let $a \in G$ and let $\delta_a$ be the unit point mass at $a$.
A {\it weighted translation} on $G$ is
a weighted convolution operator $T_{a,w}: L^\Phi(G)\longrightarrow L^\Phi(G)$ defined by
$$T_{a,w}(f)= wT_a(f) \qquad (f \in L^\Phi(G))$$
where $w$ is a weight on $G$ and $T_a(f)=f*\delta_a\in L^\Phi(G)$ is
the convolution:
$$(f*\delta_a)(x)= \int_{y\in G} f(xy^{-1})\delta_a(y) = f(xa^{-1})
\qquad (x\in G).$$
If $w^{-1}\in L^{\infty}(G)$, then we can define a self-map $S_{a,w}$ on $L^\Phi(G)$  by
$$S_{a,w}(h) = \frac{h}{w}* \delta_{a^{-1}} \qquad (h \in L^\Phi(G))$$ so that
$$T_{a,w}S_{a,w}(h) = h \qquad (h \in L^\Phi(G)).$$
In what follows, we assume $w,w^{-1}\in L^\infty(G)$.

We note that if $\|w\|_{\infty}\leq1$, then $\|T_{a,w}\|\leq1$ and $T_{a,w}$ is never transitive (hypercyclic).
Also, $T_{a,w}$ is not transitive if $a$ is a torsion element of $G$ in \cite{aa17}.
An element $a$ in a group $G$ is called a {\it torsion element} if
it is of finite order. In a locally compact group $G$, an element $a\in G$ is called {\it
periodic} (or {\it compact}) in \cite{rossbook} if the closed subgroup $G(a)$
generated by $a$ is compact. We call an element in
$G$ {\it aperiodic} if it is not periodic. For discrete groups,
periodic and torsion elements are identical.

It was proved in \cite{cc11} that an element $a\in G$ is aperiodic if, and only if, for any compact set $K\subset G$, there exists some
$M\in\Bbb N$ such that $K\cap Ka^{\pm n}=\emptyset$ for all $n>M$.
We will make use of the property of aperiodicity to achieve our goal. We note that \cite{cc11} in many familiar non-discrete groups, including the additive group $\Bbb R^d$, the Heisenberg group and the affine group, all elements except the identity are aperiodic.

In Section 2, we will characterize disjoint transitivity of powers of weighted translations, generated by the weights and an aperiodic element, on $L^\Phi(G)$ in terms of the weights, the Haar measure and the aperiodic element. Applying the result on disjoint transitivity, the characterization
for powers of weighted translations to be disjoint chaotic follows.

\section{disjoint transitivity}

Before proving the result, we recall some useful observations on the norm $N_\Phi$.
Let $B\subset G$ be a Borel set with $\lambda(B)>0$, and let $\chi_B$ be the characteristic function
of $B$. Then by a simple computation,
$$N_\Phi(\chi_B)=\frac{1}{\Phi^{-1}(\frac{1}{\lambda(B)})}$$
where $\Phi^{-1}(t)$ is the modulus of the preimage of a singleton $t$ under $\Phi$.
Besides, the norm of a function $f\in L^\Phi(G)$ is invariant under the translation by an element $a\in G$.
We include the proof for completeness.

\begin{lemma}$($\cite[Lemma 1]{chen-orlicz}\label{invariant}$)$
Let $G$ be a locally compact group, and let $a\in G$. Let $\Phi$ be a Young function, and let $f\in L^\Phi(G)$.
Then we have
$$N_\Phi(f)=N_\Phi(f*\delta_a).$$
\end{lemma}
\begin{proof}
By the right invariance of the Haar measure $\lambda$ and the definition of the norm
$$N_\Phi(f)=\inf\left\{k>0:\int_G\Phi\left(\frac{|f|}{k}\right)d\lambda\leq1\right\},$$
we have
\begin{eqnarray*}
N_\Phi(f*\delta_a)&=&\inf\left\{k>0:\int_G\Phi\left(\frac{|f*\delta_a|}{k}\right)d\lambda\leq1\right\}\\
&=&\inf\left\{k>0:\int_G\Phi\left(\frac{1}{k}\left|\int_Gf(xy^{-1})d\delta_a(y)\right|\right)d\lambda(x)\leq1\right\}\\
&=&\inf\left\{k>0:\int_G\Phi\left(\frac{1}{k}|f(xa^{-1})|\right)d\lambda(x)\leq1\right\}\\
&=&\inf\left\{k>0:\int_G\Phi\left(\frac{1}{k}|f(t)|\right)d\lambda(t)\leq1\right\}=N_\Phi(f)
\end{eqnarray*}
where $t=xa^{-1}$ and $d\lambda(t)=d\lambda(xa^{-1})=d\lambda(x)$.
\end{proof}

Now we are ready to give and prove the main result.

\begin{theorem}\label{disjointtransitive}
Let $G$ be a locally compact group, and let $a\in G$ be an aperiodic element. Let $w$ be a weight on $G$, and let
$\Phi$ be a Young function. Given some $L\geq2$, let $T_l=T_{a,w_l}$ be a
weighted translation on $L^{\Phi}(G)$, generated by $a$ and
a weight $w_l$ for $1\leq l\leq L$. For $1\leq r_1< r_2< \cdot\cdot\cdot < r_L$,
the following conditions are equivalent.
\begin{enumerate}
\item[{\rm(i)}] $T_1^{r_1}, T_2^{r_2},\cdot\cdot\cdot, T_L^{r_L}$ are disjoint topologically transitive on $L^{\Phi}(G)$.
\item[{\rm(ii)}] For each compact subset $K \subset G$ with $\lambda(K) >0$, there is a
sequence of Borel sets $(E_{k})$ in $K$ such that
$\displaystyle\lambda(K) = \lim_{k \rightarrow \infty}
\lambda(E_{k})$ and both sequences
$$\varphi_{l,n}:=\prod_{j=1}^{n}w_l\ast\delta_{a^{-1}}^{j}\ \ \ \  and \ \ \
\ \ {\widetilde \varphi_{l,n}}:=\left(\prod_{j=0}^{n-1}w_l\ast\delta_{a}^{j}\right)^{-1}$$
admit respectively subsequences $(\varphi_{l,r_ln_{k}})$ and $(\widetilde{\varphi}_{l,r_ln_k})$ satisfying
$($ for $1\leq l \leq L$ $)$
$$\lim_{k\rightarrow \infty}\|\varphi_{l,r_ln_{k}}|_{_{E_{k}}}\|_\infty=
\lim_{k\rightarrow \infty}\|{\widetilde \varphi_{l,r_ln_{k}}}|_{_{E_{k}}}\|_\infty=0$$
and $($ for $1\leq s < l \leq L$ $)$
$$\lim_{k\rightarrow \infty}\left\|\frac{{\widetilde \varphi_{s,(r_l-r_s)n_{k}}}\cdot{\widetilde \varphi_{l,r_ln_{k}}}}
{{\widetilde \varphi_{s,r_ln_{k}}}}\big|_{_{E_{k}}}\right\|_\infty
=\lim_{k\rightarrow \infty}\left\|\frac{\varphi_{l,(r_l-r_s)n_{k}}\cdot{\widetilde \varphi_{s,r_sn_{k}}}}
{{\widetilde \varphi_{l,r_sn_{k}}}}\big|_{_{E_{k}}}\right\|_\infty=0.$$
\end{enumerate}
\end{theorem}
\begin{proof}
(i) $\Rightarrow$ (ii). Let $T_1^{r_1}, T_2^{r_2},\cdot\cdot\cdot, T_N^{r_N}$ be disjoint transitive. Let
$K\subset G$ be a compact set with $\lambda (K)>0$. By aperiodicity of $a$, there is some $M$ such that $K\cap Ka^{\pm n}=\emptyset$
for all $n>M$.

Let $\varepsilon\in(0,1)$, and let $\chi_{K}\in L^\Phi(G)$ be the characteristic function of $K$.
By the assumption of disjoint transitivity, there exist a vector $f \in L^{\Phi}(G)$ and some $m>M$ such that
$$N_{\Phi}(f-\chi_K)< \varepsilon^2 \quad \mbox{and} \qquad N_{\Phi}(T_l^{r_lm}f-\chi_{K})< \varepsilon^2$$
for $l=1,2,\ldots,L$. Let
$$A=\{x\in K:|f(x)-1|\geq\varepsilon\}\qquad \mbox{and}\qquad B_{l,m}=\{x\in K:|T_l^{r_lm}f(x)-1|\geq\varepsilon\}.$$
Then
$$|f(x)|>1-\varepsilon \qquad ( x\in K\setminus A)\qquad \mbox{and}\qquad |T_l^{r_lm}f(x)|>1-\varepsilon \qquad ( x\in K\setminus B_{l,m}).$$
In both cases, $\lambda(A)<\frac{1}{\Phi(\frac{1}{\varepsilon})}$ and $\lambda(B_{l,m})<\frac{1}{\Phi(\frac{1}{\varepsilon})}$.
Indeed,
\begin{eqnarray*}
\varepsilon^2 &>& N_\Phi(f-\chi_K)\\
&\geq& N_\Phi\left(\chi_K(f-1)\right)\\
&\geq& N_\Phi\left(\chi_{A}(f-1)\right)\\
&\geq& N_\Phi(\chi_{A}\varepsilon)\\
&=& \frac{\varepsilon}{\Phi^{-1}(\frac{1}{\lambda(A)})}
\end{eqnarray*}
and
\begin{eqnarray*}
\varepsilon^2 &>& N_\Phi(T_l^{r_lm}f-\chi_K)\\
&\geq& N_\Phi\left(\chi_K(T_l^{r_lm}f-1)\right)\\
&\geq& N_\Phi\left(\chi_{B_{l,m}}(T_l^{r_lm}f-1)\right)\\
&\geq& N_\Phi(\chi_{B_{l,m}}\varepsilon)\\
&=& \frac{\varepsilon}{\Phi^{-1}(\frac{1}{\lambda(B_{l,m})})}.
\end{eqnarray*}
Similarly, put
$$C_{l,m}=\{x\in K\setminus A:\varphi_{l,r_lm}(x)\geq \varepsilon\}\qquad \mbox{and}\qquad
D_{l,m}=\{x\in K\setminus B_{l,m}:\widetilde{\varphi}_{l,r_lm}(x) \geq\varepsilon\}.$$
Then
$$\varphi_{l,r_lm}(x) < \varepsilon\ \ (x\in K\setminus (A\cup C_{l,m}))\qquad \mbox{and}\qquad
\widetilde{\varphi}_{l,r_lm}(x) < \varepsilon\ \ (x\in K\setminus (B_{l,m}\cup D_{l,m})).$$
Moreover, by Lemma \ref{invariant}, $K\cap Ka^{\pm m}=\emptyset$ and the right invariance of the Haar measure $\lambda$,
we arrive at
\begin{eqnarray*}
\varepsilon^2 &>& N_\Phi(T_l^{r_lm}f-\chi_K)\\
&\geq& N_\Phi\left(\chi_{C_{l,m}a^{r_lm}}(T_l^{r_lm}f-0)\right)\\
&=& N_\Phi\left(\chi_{C_{l,m}a^{r_lm}}(\prod_{j=0}^{r_lm-1}w_l*\delta_{a}^j)(f*\delta_{a^{r_lm}})\right)\\
&=& N_\Phi\left(\chi_{C_{l,m}}(\prod_{j=1}^{r_lm}w_l*\delta_{a^{-1}}^j)f\right)\\
&=& N_\Phi(\chi_{C_{l,m}}\varphi_{l,r_lm}f)\\
&>& \frac{\varepsilon(1-\varepsilon)}{\Phi^{-1}(\frac{1}{\lambda(C_{l,m})})}
\end{eqnarray*}
and
\begin{eqnarray*}
\varepsilon^2 &>& N_\Phi(f-\chi_K)\\
&\geq& N_\Phi\left(\chi_{D_{l,m}a^{-r_lm}}(S_l^{r_lm}T_l^{r_lm}f-0)\right)\\
&=& N_\Phi\left(\chi_{D_{l,m}a^{-r_lm}}(\prod_{j=1}^{r_lm}w_l*\delta_{a^{-1}}^j)^{-1}((T_l^{r_lm}f)*\delta_{a^{-r_lm}})\right)\\
&=& N_\Phi\left(\chi_{D_{l,m}}(\prod_{j=0}^{r_lm-1}w_l*\delta_{a}^j)^{-1}(T_l^{r_lm}f)\right)\\
&=& N_\Phi\left(\chi_{D_{l,m}}\widetilde{\varphi}_{l,r_lm}(T_l^{r_lm}f)\right)\\
&>& \frac{\varepsilon(1-\varepsilon)}{\Phi^{-1}(\frac{1}{\lambda(D_{l,m})})},
\end{eqnarray*}
which implies
$\lambda(C_{l,m})<\frac{1}{\Phi(\frac{1-\varepsilon}{\varepsilon})}$
and
$\lambda(D_{l,m})<\frac{1}{\Phi(\frac{1-\varepsilon}{\varepsilon})}.$
Hence we have the preliminary estimates for the two sequences $\varphi_{l,r_lm}, \widetilde{\varphi}_{l,r_lm}$,
and the measures of the four sets $A, B_{l,m}, C_{l,m}, D_{l,m}$.

Next, we will consider the other two weight conditions for $1\leq s< l\leq L$.
Since
$$B_{l,m}=\{x\in K:|w_l(x)w_l(xa^{-1})\cdot\cdot\cdot w_l(xa^{-(r_lm-1)})f(xa^{-r_lm})-1|\geq\varepsilon\},$$
we have
$$|w_l(x)w_l(xa^{-1})\cdot\cdot\cdot w_l(xa^{-(r_lm-1)})f(xa^{-r_lm})|>1-\varepsilon \qquad ( x\in K\setminus B_{l,m}).$$
For $1\leq s< l\leq L$, let
$$F_{s,m}=\{x\in G\setminus K: |w_s(x)w_s(xa^{-1})\cdot\cdot\cdot w_s(xa^{-(r_sm-1)})f(xa^{-r_sm})|\geq\varepsilon\}.$$
Then
$$|w_s(x)w_s(xa^{-1})\cdot\cdot\cdot w_s(xa^{-(r_sm-1)})f(xa^{-r_sm})|<\varepsilon\ \ \mbox{on}\ \ Ka^{-(r_l-r_s)m}\setminus F_{s,m}\subset G\setminus K$$
which says
$$|w_s(xa^{-(r_l-r_s)m})\cdot\cdot\cdot w_s(xa^{-(r_lm-1)})f(xa^{-r_lm})|<\varepsilon\ \ \mbox{on}\ \ K\setminus F_{s,m}a^{(r_l-r_s)m}.$$
Moreover, $\lambda(F_{s,m})<\frac{1}{\Phi(\frac{1}{\varepsilon})}$ by
\begin{eqnarray*}
\varepsilon^{2}&>& N_\Phi(T_{s}^{r_sm}f-\chi_{K})\\
&\geq& \Phi(\chi_{F_{s,m}}(T_{s}^{r_sm}f-0))\\
&=& \Phi(\chi_{F_{s,m}}w_s(w_s*\delta_a)\cdot\cdot\cdot(w_s*\delta_a^{r_sm-1})(f*a^{r_sm}))\\
&\geq& \Phi(\chi_{F_{s,m}}\varepsilon)\\
&=& \frac{\varepsilon}{\Phi^{-1}(\frac{1}{\lambda(F_{s,m})})}.
\end{eqnarray*}
Therefore, for $1\leq s< l\leq L$ and $x\in K\setminus (F_{s,m}a^{(r_l-r_s)m}\cup B_{l,m})$,
\begin{eqnarray*}
&&\frac{ w_s(xa^{-(r_l-r_s)m})\cdot\cdot\cdot w_s(xa^{-(r_lm-1)})}{w_l(x)w_l(xa^{-1})\cdot\cdot\cdot w_l(xa^{-(r_lm-1)})}\\
&=&\frac{ w_s(xa^{-(r_l-r_s)m})\cdot\cdot\cdot w_s(xa^{-(r_lm-1)})|f(xa^{-r_lm})|}{w_l(x)w_l(xa^{-1})\cdot\cdot\cdot w_l(xa^{-(r_lm-1)})|f(xa^{-r_lm})|}<\frac{\varepsilon}{1-\varepsilon}.
\end{eqnarray*}
That is,
$$\frac{{\widetilde \varphi_{s,(r_l-r_s)m}}(x)\cdot{\widetilde \varphi_{l,r_lm}}(x)}
{{\widetilde \varphi_{s,r_lm}}(x)}<\frac{\varepsilon}{1-\varepsilon} \quad \mbox{on}\quad K\setminus (F_{s,m}a^{(r_l-r_s)m}\cup B_{l,m}).$$
Similarly, by
\begin{equation*}
|w_l(x)w_l(xa^{-1})\cdot\cdot\cdot w_l(xa^{-(r_lm-1)})f(xa^{-r_lm})|<\varepsilon \quad \mbox{on} \quad Ka^{-(r_s-r_l)m}\setminus F_{l,m},
\end{equation*}
one has
$$|w_l(xa^{-(r_s-r_l)m})\cdot\cdot\cdot w_l(xa^{-(r_sm-1)})f(xa^{-r_sm})|<\delta \quad \mbox{on} \quad K\setminus F_{l,m}a^{(r_s-r_l)m}.$$
Therefore, for $1\leq s< l\leq L$ and $x\in K\setminus (F_{l,m}a^{(r_s-r_l)m}\cup B_{s,m})$,
\begin{eqnarray*}
&&\frac{ w_l(xa^{-(r_s-r_l)m})\cdot\cdot\cdot w_l(xa^{-(r_sm-1)})}{w_s(x)w_s(xa^{-1})\cdot\cdot\cdot
w_s(xa^{-(r_sm-1)})}\\
&=&\frac{ w_l(xa^{-(r_s-r_l)m})\cdot\cdot\cdot
w_l(xa^{-(r_sm-1)})|f(xa^{-r_sm})|}{w_s(x)w_s(xa^{-1})\cdot\cdot\cdot
w_s(xa^{-(r_sm-1)})|f(xa^{-r_sm})|}<\frac{\varepsilon}{1-\varepsilon}.
\end{eqnarray*}
Hence
$$\frac{\varphi_{l,(r_l-r_s)m}(x)\cdot{\widetilde \varphi_{s,r_sm}}(x)}
{{\widetilde \varphi_{l,r_sm}}(x)}<\varepsilon\quad\mbox{on}\quad K\setminus (F_{l,m}a^{(r_s-r_l)m}\cup B_{s,m}).$$
Finally, put
$$E_{m}=(K\setminus A) \setminus \displaystyle\bigcup_{1\leq l\leq L} (B_{l,m}\cup C_{l,m}\cup D_{l,m})
\setminus \displaystyle\bigcup_{1\leq s<l\leq L}(F_{s,m}a^{(r_l-r_s)m}\cup F_{l,m}a^{(r_s-r_l)m}).$$
Then
$$\lambda(K\setminus E_{m})<\frac{1+L^2}{\Phi(\frac{1}{\varepsilon})}+\frac{2L}{\Phi(\frac{1-\varepsilon}{\varepsilon})},\qquad
\|\varphi_{l,r_lm}|_{_{E_{m}}}\|_\infty<\varepsilon,\qquad \|{\widetilde \varphi_{l,r_lm}}|_{_{E_{m}}}\|_\infty<\varepsilon,$$
and
$$\left\|\frac{{\widetilde \varphi_{s,(r_l-r_s)m}}\cdot{\widetilde \varphi_{l,r_lm}}}
{{\widetilde \varphi_{s,r_lm}}}\big|_{_{E_{m}}}\right\|_\infty<\frac{\varepsilon}{1-\varepsilon},\qquad
\left\|\frac{\varphi_{l,(r_l-r_s)m}\cdot{\widetilde \varphi_{s,r_sm}}}
{{\widetilde \varphi_{l,r_sm}}}\big|_{_{E_{m}}}\right\|_\infty<\frac{\varepsilon}{1-\varepsilon},$$
proving the condition (ii).

(ii) $\Rightarrow$ (i). We show that $T_1^{r_1}, T_2^{r_2},\cdot\cdot\cdot, T_L^{r_L}$ are
disjoint topologically transitive.
For $1\leq l\leq L$, let $U$ and $V_l$ be non-empty open subsets of $L^\Phi(G)$.
Since the space $C_c(G)$ of continuous functions on $G$ with
compact support is dense in $L^\Phi(G)$, we can pick $f,g_l\in C_c(G)$
with $f\in U$ and $g_l\in V_l$ for $l=1,2,\cdot\cdot\cdot,L$. Let $K$ be the union of the compact
supports of $f$ and all $g_l$.

Let $E_k\subset K$ and the sequences $(\varphi_{l,n}), ({\widetilde \varphi_{l,n}})$
satisfy condition (ii). Then we may assume $\varphi_{l,r_ln_k}<\frac{\frac{1}{2^k}}{\|f\|_\infty}$
on $E_k$ for $l=1,2,\cdot\cdot\cdot,L$.
Hence, by Lemma \ref{invariant}, we first observe for $1\leq l\leq L$,
\begin{eqnarray*}
&&N_{\Phi}\left(T_l^{r_ln_k}(f\chi_{E_k})\right)\\
&=&N_{\Phi}\left((\prod_{j=0}^{r_ln_k-1}w_l*\delta_{a}^j)(f*\delta_{a^{r_ln_k}})(\chi_{E_k}*\delta_{a^{r_ln_k}})\right)\\
&=&N_{\Phi}\left((\prod_{j=1}^{r_ln_k}w_l*\delta_{a^{-1}}^j)f\chi_{E_k}\right)\\
&=&N_{\Phi}\left(\varphi_{l,r_ln_k}f\chi_{E_k}\right)\\
&<&\frac{\|f\|_{\infty}\frac{\frac{1}{2^k}}{\|f\|_{\infty}}}{\Phi^{-1}(\frac{1}{\lambda(E_k)})}\rightarrow 0
\end{eqnarray*}
as $k\rightarrow \infty$. Likewise, by the sequence $(\widetilde \varphi_{l,r_ln_k})$,
\begin{eqnarray*}
&&\lim_{k\rightarrow \infty}N_{\Phi}\left(S_l^{r_ln_k}(f\chi_{E_k})\right)\\
&=&\lim_{k\rightarrow \infty}N_{\Phi}\left((\prod_{j=1}^{r_ln_k}w_l*\delta_{a^{-1}}^j)^{-1}(f*\delta_{a^{-ln_k}})(\chi_{E_k}*\delta_{a^{-ln_k}})\right)\\
&=&\lim_{k\rightarrow \infty}N_{\Phi}\left((\prod_{j=0}^{r_ln_k-1}w_l*\delta_{a}^j)^{-1}f\chi_{E_k}\right)\\
&=&\lim_{k\rightarrow \infty}N_{\Phi}\left(\widetilde{\varphi}_{l,r_ln_k}f\chi_{E_k}\right)=0
\end{eqnarray*}
for $l=1,2,...,L$. Similarly, for $1\leq s<l\leq L$,
\begin{eqnarray*}
&& N_\Phi\left(T_l^{r_ln_k}(S_s^{r_sn_k}(g_s\chi_{E_k}))\right)\\
&=& N_\Phi\left( w_l(w_l*\delta_a^1)\cdot\cdot\cdot (w_l*\delta_a^{r_ln_k-1})((S_s^{r_sn_k}(g_s\chi_{E_k}))*\delta_a^{r_ln_{k}})\right)\\
&=& N_\Phi\left( \frac{w_l(w_l*\delta_a^1)\cdot\cdot\cdot (w_l*\delta_a^{r_ln_k-1})}
{(w_s*\delta_a^{r_ln_k-1})(w_s*\delta_a^{r_ln_k-2})\cdot\cdot\cdot (w_s\delta_a^{r_ln_k-r_sn_k})}
((g_s\chi_{E_k})*\delta_a^{r_ln_k}*\delta_a^{-r_sn_k})\right)\\
&=& N_\Phi\left( \frac{(w_l*\delta_a^{(r_s-r_l)n_k})(w_l*\delta_a^{(r_s-r_l)n_k+1})\cdot\cdot\cdot (w_l*\delta_a^{r_sn_k-1})}
{(w_s*\delta_a^{r_sn_k-1})(w_s*\delta_a^{r_sn_k-2})\cdot\cdot\cdot (w_s)}
(g_s\chi_{E_k})\right)\\
&=& N_\Phi\left(\frac{\varphi_{l,(r_l-r_s)n_k}\cdot\widetilde{\varphi}_{s,r_sn_k}}{\widetilde{\varphi}_{l,r_sn_k}}(g_s\chi_{E_k})\right)\rightarrow 0
\end{eqnarray*}
as $k\rightarrow \infty$, and
\begin{eqnarray*}
&& N_\Phi\left(T_s^{r_sn_k}(S_l^{r_ln_k}(g_l\chi_{E_k}))\right)\\
&=& N_\Phi\left( w_s(w_s*\delta_a^1)\cdot\cdot\cdot (w_s*\delta_a^{r_sn_k-1})((S_l^{r_ln_k}(g_l\chi_{E_k}))*\delta_a^{r_sn_{k}})\right)\\
&=& N_\Phi\left( \frac{w_s(w_s*\delta_a^1)\cdot\cdot\cdot (w_s*\delta_a^{r_sn_k-1})}
{(w_l*\delta_a^{r_sn_k-1})(w_l*\delta_a^{r_sn_k-2})\cdot\cdot\cdot (w_l\delta_a^{r_sn_k-r_ln_k})}
((g_l\chi_{E_k})*\delta_a^{r_sn_k}*\delta_a^{-r_ln_k})\right)\\
&=& N_\Phi\left( \frac{(w_s*\delta_a^{(r_l-r_s)n_k})(w_s*\delta_a^{(r_l-r_s)n_k+1})\cdot\cdot\cdot (w_s*\delta_a^{r_ln_k-1})}
{(w_l*\delta_a^{r_ln_k-1})(w_l*\delta_a^{r_ln_k-2})\cdot\cdot\cdot (w_l)}
(g_l\chi_{E_k})\right)\\
&=& N_\Phi\left(\frac{\varphi_{s,(r_s-r_l)n_k}\cdot\widetilde{\varphi}_{l,r_ln_k}}{\widetilde{\varphi}_{s,r_ln_k}}(g_l\chi_{E_k})\right)\rightarrow 0
\end{eqnarray*}
as $k\rightarrow \infty$ for $1\leq s<l\leq L$.
Now for each $k\in \Bbb{N}$, put
$$v_k=f\chi_{E_k}+S^{r_1n_k}_{1}(g_1\chi_{E_k})+S^{r_2n_k}_{2}(g_2\chi_{E_k})+\cdot\cdot\cdot+S^{r_Ln_k}_{L}(g_L\chi_{E_k}).$$
Then
$$ N_\Phi\left( v_k-f \right) \leq N_\Phi\left( f\chi_{K\setminus E_k} \right)+\sum_{l=1}^{L} N_\Phi\left( S^{r_ln_k}_{l}(g_l\chi_{E_k})\right)$$
and
\begin{eqnarray*}
&& N_\Phi\left(T^{r_ln_k}_{l}v_k-g_l \right)\\
&\leq&  N_\Phi\left( T^{r_ln_k}_{l}(f\chi_{E_k})  \right) + N_\Phi\left( T^{r_ln_k}_{l}S_1^{r_1n_k}(g_1\chi_{E_k})  \right)
+\cdot\cdot\cdot+ N_\Phi\left(  T^{r_ln_k}_{l}S_{l-1}^{r_{l-1}n_k}(g_{l-1}\chi_{E_k}) \right)   \\
&+& N_\Phi\left( g_l\chi_{K\setminus E_k} \right) +  N_\Phi\left( T^{r_ln_k}_{l}S_{l+1}^{r_{l+1}n_k}(g_{l+1}\chi_{E_k})  \right)
+\cdot\cdot\cdot+ N_\Phi\left( T^{r_ln_k}_{l}S_{L}^{r_{L}n_k}(g_{L}\chi_{E_k})  \right).
\end{eqnarray*}
Hence
$\lim_{k\rightarrow \infty}v_{k}=f$ and $\lim_{k\rightarrow \infty}T_l^{r_ln_k}v_k=g_l$ for $l=1,2,\cdot\cdot\cdot,L$.
Therefore
$$\emptyset\neq U\cap T_1^{-{r_1n_k}}(V_1)\cap T_2^{-{r_2n_k}}(V_2)\cap\cdot\cdot\cdot\cap T_L^{-{r_Ln_k}}(V_L).$$
\end{proof}

As in \cite[Example 2.3, 2.4, 2.5]{chen171}, it is not difficult to find the weight satisfying the above weight condition
on various locally compact groups (for instance, $G=\Bbb{Z}$ or $G=\Bbb{R}$). For completeness, we include one example of Heisenberg groups only.

\begin{example}\label{example}
Let $$G=\Bbb{H}:=\left\{\left(\begin{matrix}
1 & x & z\\
0 & 1 & y\\
0 & 0 & 1
\end{matrix}\right):x,y,z\in \Bbb{R}\right\}$$
be the Heisenberg group which is neither abelian nor compact. For
convenience, an element in $G$ is written as $(x,y,z)$. Let $(x,y,z),(x',y',z')\in \Bbb{H}$.
Then the multiplication is given by $$(x,y,z)\cdot(x',y',z')=(x+x',y+y',z+z'+xy')$$ and
$$(x,y,z)^{-1}=(-x,-y,xy-z).$$

Let $a=(1,0,2)\in \Bbb{H}$ which is aperiodic. Given $L\geq2$, let $w_l$ be a weight on $\Bbb{H}$ for $l=1,2,\cdot\cdot\cdot,L$.
Then $a^{-1}=(-1,0,-2)$ and the weighted translation $T_{(1,0,2),w_l}$
on $L^\Phi(\Bbb{H})$ is defined by
$$T_{(1,0,2),w_l}f(x,y,z)=w_l(x,y,z)f(x-1,y,z-2)\qquad(f\in L^\Phi(\Bbb{H})).$$
By Theorem \ref{disjointtransitive}, for $1\leq r_1< r_2< \cdot\cdot\cdot < r_L$,
the operators $T_{(1,0,2),w_1}^{r_1}, T_{(1,0,2),w_2}^{r_2},\cdot\cdot\cdot, T_{(1,0,2),w_L}^{r_L}$
are disjoint topologically transitive on $L^\Phi(\Bbb{H})$ if
given $\varepsilon>0$ and a compact subset $K$ of $\Bbb{H}$,
there exists a positive integer $n$ such that for $(x,y,z)\in K$, we have,
for $1\leq l\leq L$,
$$\varphi_{l,r_ln}(x,y,z)=\prod_{j=1}^{r_ln}w_l\ast\delta^j_{(1,0,2)^{-1}}(x,y,z)=\prod_{j=1}^{r_ln}w_l(x+j,y,z+2j)<\varepsilon$$
and
$$\widetilde{\varphi}^{^{-1}}_{l,r_ln}(x,y,z)=\prod_{j=0}^{r_ln-1}w_l\ast\delta^j_{(1,0,2)}(x)=\prod_{j=0}^{r_ln-1}w_l(x-j,y,z-2j)>\frac{1}{\varepsilon},$$
and for $1\leq s< l\leq L$,
\begin{eqnarray*}
&&\frac{{\widetilde \varphi_{s,(r_l-r_s)n}}(x,y,z)\cdot{\widetilde \varphi_{l,r_ln}}(x,y,z)}{{\widetilde \varphi_{s,r_ln}}(x,y,z)}\\
&=&\frac{\prod_{j=0}^{r_ln-1}w_s(x-j,y,z-2j)}{\prod_{j=0}^{(r_l-r_s)n-1}w_s(x-j,y,z-2j)\cdot\prod_{j=0}^{r_ln-1}w_l(x-j,y,z-2j)}\\
&=&\frac{\prod_{j=(r_l-r_s)n}^{r_ln-1}w_s(x-j,y,z-2j)}{\prod_{j=0}^{r_ln-1}w_l(x-j,y,z-2j)}<\varepsilon
\end{eqnarray*}
and
\begin{eqnarray*}
&&\frac{\varphi_{l,(r_l-r_s)n}(x,y,z)\cdot{\widetilde \varphi_{s,r_sn}}(x,y,z)}{{\widetilde \varphi_{l,r_sn}}(x,y,z)}\\
&=&\frac{\prod_{j=1}^{(r_l-r_s)n}w_l(x+j,y,z+2j)\cdot\prod_{j=0}^{r_sn-1}w_l(x-j,y,z-2j)}{\prod_{j=0}^{r_sn-1}w_s(x-j,y,z-2j)}\\
&=&\frac{\prod_{j=-(r_l-r_s)n}^{r_sn-1}w_l(x-j,y,z-2j)}{\prod_{j=0}^{r_sn-1}w_s(x-j,y,z-2j)}<\varepsilon.
\end{eqnarray*}
One can obtain the required weight conditions by defining $w:\Bbb{H}\rightarrow (0,\infty)$ as follows:
$$w_l(x,y,z)=\left\{
\begin{array}{ll}
\frac{1}{2}& \mbox{ if } z\geq 1\\\\
\frac{1}{2^z}& \mbox{ if } -1<z< 1\\\\
2 & \mbox{ if } z\leq-1
\end{array}
\right.$$
for $1\leq l\leq L$.
\end{example}

If the weighted translation $T_{a,w_l}$ is generated by the same weight $w_l:=w$ for $l=1,2,\cdot\cdot\cdot,L$ in
Theorem \ref{disjointtransitive}, then we have a simpler characterization of disjoint transitivity.

\begin{corollary}\label{thesmaeweight}
Let $G$ be a locally compact group, and let $a\in G$ be an aperiodic element. Let
$\Phi$ be a Young function. Given some $L\geq2$, let $T=T_{a,w}$ be a
weighted translation on $L^{\Phi}(G)$, generated by $a$ and
a weight $w$. For $1\leq r_1< r_2< \cdot\cdot\cdot < r_L$,
the following conditions are equivalent.
\begin{enumerate}
\item[{\rm(i)}] $T^{r_1}, T^{r_2},\cdot\cdot\cdot, T^{r_L}$ are disjoint topologically transitive on $L^{\Phi}(G)$.
\item[{\rm(ii)}] For each compact subset $K \subset G$ with $\lambda(K) >0$, there is a
sequence of Borel sets $(E_{k})$ in $K$ such that
$\displaystyle\lambda(K) = \lim_{k \rightarrow \infty}
\lambda(E_{k})$ and both sequences
$$\varphi_{n}:=\prod_{j=1}^{n}w\ast\delta_{a^{-1}}^{j}\ \ \ \  and \ \ \
\ \ {\widetilde \varphi_{n}}:=\left(\prod_{j=0}^{n-1}w\ast\delta_{a}^{j}\right)^{-1}$$
admit respectively subsequences $(\varphi_{r_ln_{k}})$ and $(\widetilde{\varphi}_{r_ln_k})$ satisfying
$($ for $1\leq l \leq L$ $)$
$$\lim_{k\rightarrow \infty}\|\varphi_{r_ln_{k}}|_{_{E_{k}}}\|_\infty=
\lim_{k\rightarrow \infty}\|{\widetilde \varphi_{r_ln_{k}}}|_{_{E_{k}}}\|_\infty=0$$
and $($ for $1\leq s < l \leq L$ $)$
$$\lim_{k\rightarrow \infty}\|{\widetilde \varphi_{(r_l-r_s)n_{k}}}|_{_{E_{k}}}\|_\infty
=\lim_{k\rightarrow \infty}\|\varphi_{(r_l-r_s)n_{k}}|_{_{E_{k}}}\|_\infty=0.$$
\end{enumerate}
\end{corollary}

\begin{remark}
Let $r_1=1, r_2=2,\cdot\cdot\cdot,r_L=L$ in Corollary \ref{thesmaeweight}. Then
$T^1, T^2,\cdot\cdot\cdot, T^L$ are disjoint transitive if, and only if,
$$\lim_{k\rightarrow \infty}\|\varphi_{ln_{k}}|_{_{E_{k}}}\|_\infty=
\lim_{k\rightarrow \infty}\|{\widetilde \varphi_{ln_{k}}}|_{_{E_{k}}}\|_\infty=0.$$
\end{remark}

By strengthening the weight condition in Theorem \ref{disjointtransitive}, we can characterize disjoint mixing weighted translations
on the Orlicz space $L^\Phi(G)$.

\begin{corollary}
Let $G$ be a locally compact group, and let $a\in G$ be an aperiodic element. Let
$\Phi$ be a Young function. Given some $L\geq2$, let $T_l=T_{a,w_l}$ be a
weighted translation on $L^{\Phi}(G)$, generated by $a$ and
a weight $w_l$ for $1\leq l\leq L$. For $1\leq r_1< r_2< \cdot\cdot\cdot < r_L$,
the following conditions are equivalent.
\begin{enumerate}
\item[{\rm(i)}] $T_1^{r_1}, T_2^{r_2},\cdot\cdot\cdot, T_L^{r_L}$ are disjoint topologically mixing on $L^{\Phi}(G)$.
\item[{\rm(ii)}] For each compact subset $K \subset G$ with $\lambda(K) >0$, there is a
sequence of Borel sets $(E_{n})$ in $K$ such that
$\displaystyle\lambda(K) = \lim_{n \rightarrow \infty}
\lambda(E_{n})$ and both sequences
$$\varphi_{l,n}:=\prod_{j=1}^{n}w_l\ast\delta_{a^{-1}}^{j}\ \ \ \  and \ \ \
\ \ {\widetilde \varphi_{l,n}}:=\left(\prod_{j=0}^{n-1}w_l\ast\delta_{a}^{j}\right)^{-1}$$
satisfy $($ for $1\leq l \leq L$ $)$
$$\lim_{n\rightarrow \infty}\|\varphi_{l,r_ln}|_{_{E_{n}}}\|_\infty=
\lim_{n\rightarrow \infty}\|{\widetilde \varphi_{l,r_ln}}|_{_{E_{n}}}\|_\infty=0$$
and $($ for $1\leq s < l \leq L$ $)$
$$\lim_{n\rightarrow \infty}\left\|\frac{{\widetilde \varphi_{s,(r_l-r_s)n}}\cdot{\widetilde \varphi_{l,r_ln}}}
{{\widetilde \varphi_{s,r_ln}}}\big|_{_{E_{n}}}\right\|_\infty
=\lim_{n\rightarrow \infty}\left\|\frac{\varphi_{l,(r_l-r_s)n}\cdot{\widetilde \varphi_{s,r_sn}}}
{{\widetilde \varphi_{l,r_sn}}}\big|_{_{E_{n}}}\right\|_\infty=0.$$
\end{enumerate}
\end{corollary}
\begin{proof}
The proof is similar to that of Theorem \ref{disjointtransitive} by using the full sequence $(n)$ instead of subsequence $(n_k)$.
\end{proof}

We end this paper by giving a sufficient and necessary condition for weighted translation on the Orlicz space $L^{\Phi}(G)$ to be disjoint chaotic.
From the observation in Remark \ref{remark}, the characterization for single operator to be chaotic together with Theorem \ref{disjointtransitive}
will describe disjoint chaos. Hence, we recall a result in \cite{chen-orlicz} below.

\begin{theorem}$($\cite[Theorem 2]{chen-orlicz}\label{chaos}$)$
Let $G$ be a locally compact group, and let $a\in G$ be an aperiodic element. Let $w$ be a weight on $G$, and let
$\Phi$ be a Young function. Let $T_{a,w}$ be a weighted translation on $L^\Phi(G)$, and let ${\cal P}(T_{a,w})$ be the set of periodic elements of $T_{a,w}$. Then the following conditions are equivalent.
\begin{enumerate}
\item[{\rm(i)}] $T_{a,w}$ is chaotic on $L^\Phi(G)$.
\item[{\rm(ii)}] ${\cal P}(T_{a,w})$ is dense in $L^\Phi(G)$.
\item[{\rm(iii)}]For each compact subset $K \subseteq G$ with $\lambda(K) >0$, there is a
sequence of Borel sets $(E_{k})$ in $K$ such that
$\displaystyle\lambda(K) = \lim_{k \rightarrow \infty}\lambda(E_k)$, and both sequences
$$\varphi_{n}:=\prod_{j=1}^{n}w\ast\delta_{a^{-1}}^{j}\ \ \ \  and \ \ \ \ \
{\widetilde \varphi_{n}}:=\left(\prod_{j=0}^{n-1}w\ast\delta_{a}^{j}\right)^{-1}$$
admit respectively subsequences $(\varphi_{n_k})$ and $(\widetilde{\varphi}_{n_k})$ satisfying
$$\lim_{k\rightarrow\infty}\left\|\sum_{l=1}^{\infty}\varphi_{ln_k}+\sum_{l=1}^{\infty}\widetilde{\varphi}_{ln_k}\Big|_{E_k}\right\|_{\infty}=0.$$
\end{enumerate}
\end{theorem}

Applying Theorem \ref{disjointtransitive} and Theorem \ref{chaos}, one can obtain the following result directly, and therefore
we omit its proof here.

\begin{theorem}\label{disjointchaotic}
Let $G$ be a locally compact group, and let $a\in G$ be an aperiodic element. Let
$\Phi$ be a Young function. Given some $L\geq2$, let $T_l=T_{a,w_l}$ be a
weighted translation on $L^{\Phi}(G)$, generated by $a$ and
a weight $w_l$ for $1\leq l\leq L$. For $1\leq r_1< r_2< \cdot\cdot\cdot < r_L$,
the following conditions are equivalent.
\begin{enumerate}
\item[{\rm(i)}] $T_1^{r_1}, T_2^{r_2},\cdot\cdot\cdot, T_L^{r_L}$ are disjoint chaotic on $L^{\Phi}(G)$.
\item[{\rm(ii)}] For each compact subset $K \subset G$ with $\lambda(K) >0$, there is a
sequence of Borel sets $(E_{k})$ in $K$ such that $\displaystyle\lambda(K) = \lim_{k \rightarrow \infty}\lambda(E_{k})$
and both sequences
$$\varphi_{l,r_ln}:=\prod_{j=1}^{r_ln}w_l\ast\delta_{a^{-1}}^{j}\ \ \ \  and
\ \ \ \ \ {\widetilde \varphi_{l,r_ln}}:=\left(\prod_{j=0}^{r_ln-1}w_l\ast\delta_{a}^{j}\right)^{-1}$$
admit respectively subsequences $(\varphi_{l,tr_ln_{k}})$ and $(\widetilde{\varphi}_{l,tr_ln_k})$ satisfying
$($ for $1\leq l \leq L$ $)$
$$\lim_{k\rightarrow\infty}\left\|\sum_{t=1}^{\infty}\varphi_{l,tr_ln_k}+\sum_{t=1}^{\infty}\widetilde{\varphi}_{l,tr_ln_k}\Big|_{E_k}\right\|_{\infty}=0$$
and
$($ for $1\leq s < l \leq L$ $)$
$$\lim_{k\rightarrow \infty}\left\|\frac{{\widetilde \varphi_{s,(r_l-r_s)n_{k}}}\cdot{\widetilde \varphi_{l,r_ln_{k}}}}
{{\widetilde \varphi_{s,r_ln_{k}}}}\big|_{_{E_{k}}}\right\|_\infty
=\lim_{k\rightarrow \infty}\left\|\frac{\varphi_{l,(r_l-r_s)n_{k}}\cdot{\widetilde \varphi_{s,r_sn_{k}}}}
{{\widetilde \varphi_{l,r_sn_{k}}}}\big|_{_{E_{k}}}\right\|_\infty=0.$$
\end{enumerate}
\end{theorem}

\begin{remark}
We note that the weight $w$ given in Example \ref{example} also satisfies the condition (ii) in Theorem \ref{disjointchaotic}.
\end{remark}

\vspace{.1in}
\end{document}